\documentclass[11pt,twoside, leqno]{article}
\UseRawInputEncoding 
\usepackage{mathrsfs}
\usepackage{amssymb}
\usepackage{amsmath}
\usepackage{amsthm}

\usepackage{amsfonts}

\usepackage{color}

\allowdisplaybreaks

\def\rr{{\mathbb R}}
\def\rn{{{\rr}^n}}
\def\zz{{\mathbb Z}}

\def\nn{{\mathbb N}}
\def\fz{\infty}

\def\ccc{{\mathbb C}}

\def\cs{{\mathcal S}}

\def\cp{{\mathcal P}}
\def\az{\alpha}

\renewcommand\tilde{\widetilde}

\def\supp{{\rm{\,supp\,}}}
\def\esup{\mathop\mathrm{\,ess\,sup\,}}
\def\einf{\mathop\mathrm{\,ess\,inf\,}}

\def\ls{\lesssim}
\def\lz{\lambda}

\def\ez{\varepsilon}

\def\vz{\varphi}

\def\sz{\sigma}

\def\wz{\widetilde}

\def\hs{\hspace{0.3cm}}

\def\r{\right}
\def\lf{\left}

\def\bint{{\ifinner\rlap{\bf\kern.30em--}
\int\else\rlap{\bf\kern.35em--}\int\fi}\ignorespaces}

\def\sbint{{\ifinner\rlap{\bf\kern.32em--}
\hspace{0.078cm}\int\else\rlap{\bf\kern.45em--}\int\fi}\ignorespaces}

\def\dsup{\displaystyle\sup}

\newtheorem{theorem}{Theorem}[section]
\newtheorem{lemma}[theorem]{Lemma}

\theoremstyle{definition}
\newtheorem{remark}[theorem]{Remark}
\newtheorem{definition}[theorem]{Definition}
\numberwithin{equation}{section}

\numberwithin{equation}{section}

\numberwithin{equation}{section}

\begin{document}

\arraycolsep=1pt

\title{\Large\bf Molecular Decomposition of Anisotropic Hardy Spaces with Variable Exponents \footnotetext{\hspace{-0.35cm} {\it 2010 Mathematics Subject Classification}.
{Primary 42B20; Secondary 42B30, 46E30.}
\endgraf{\it Key words and phrases.} Anisotropy, Hardy space, molecule, Calder\'on-Zygmund operator.
\endgraf This work is partially supported by the National
Natural Science Foundation of China (Grant Nos. 11861062 \& 11661075)
and the Natural Science Foundation of Xinjiang
Uygur Autonomous Region (Grant Nos. 2019D01C049, 62008031, 042312023).
 \endgraf $^\ast$\,Corresponding author
}}
\author{Wenhua Wang, Xiong Liu, Aiting Wang and Baode Li\,$^\ast$}
\date{  }
\maketitle

\vspace{-0.8cm}

\begin{center}
\begin{minipage}{13cm}\small
{\noindent{\bf Abstract} \
Let $A$ be an expansive dilation on $\mathbb{R}^n$,
and $p(\cdot):\mathbb{R}^n\rightarrow(0,\,\infty)$ be a variable exponent function satisfying the globally log-H\"{o}lder continuous condition. Let $H^{p(\cdot)}_A({\mathbb {R}}^n)$ be the variable anisotropic Hardy space defined via the non-tangential grand
maximal function. In this paper, the authors establish
its molecular decomposition, which is still new even in the classical isotropic setting (in the case $A:=2\mathrm I_{n\times n}$).
As applications, the authors obtain
the boundedness of anisotropic Calder\'on-Zygmund operators
from $H^{p(\cdot)}_{A}(\mathbb{R}^n)$ to $L^{p(\cdot)}(\mathbb{R}^n)$ or from $H^{p(\cdot)}_{A}(\mathbb{R}^n)$ to itself. }
\end{minipage}
\end{center}

\section{Introduction}
\hskip\parindent
As a good substitute of the Lebesgue space $L^p(\rn)$ when $p\in(0,1]$,
Hardy space $H^p(\rn)$ plays an important role in various fields of analysis and partial differential equations; see, for examples, \cite{fs72,s93,s60,sa89,t17, zz18}.
On the other hand, variable exponent function spaces have their applications in fluid dynamics \cite{am02}, image processing \cite{clr06}, partial differential equations and variational calculus \cite{f07,t17}.

Let $p(\cdot):\rn\rightarrow(0,\,\infty)$ be a variable exponent function satisfying the globally log-H\"{o}lder continuous condition (see Section \ref{s2} below for its definition). Recently, Nakai and Sawano \cite{ns12} introduced
the variable Hardy space $H^{p(\cdot)}({\mathbb {R}}^n)$, via the radial grand maximal function, and then obtained some real-variable characterizations of the space, such as the characterizations in terms of the atomic and the molecular decompositions. Moreover, they obtained the boundedness of $\delta$-type Calder\'on-Zygmund operators from $H^{p(\cdot)}(\rn)$ to $L^{p(\cdot)}(\rn)$ or from $H^{p(\cdot)}(\rn)$ to itself. Then Sawano \cite{s13}, Yang et al. \cite{yz16} and Zhuo et al. \cite{zy16} further contributed to the theory.

Very recently, Liu et al. \cite{lyy17x}
introduced the variable anisotropic Hardy space $H^{p(\cdot)}_A({\mathbb {R}}^n)$ associated with a general expansive matrix $A$, via the non-tangential grand maximal function, and then established its various real-variable characterizations of $H_A^{p(\cdot)}$, respectively, in terms of the atomic characterization and  the Littlewood-Paley characterization. Nevertheless, the molecular decomposition of $H^{p(\cdot)}_A({\mathbb {R}}^n)$ has
not been established till now. Once its molecular decomposition is established,
it can be conveniently used to prove the boundedness of many important operators on the space
$H^{p(\cdot)}_A({\mathbb {R}}^n)$, for example, one of the most famous operator in harmonic analysis, Calder\'on-Zygmund operators.
 To complete the theory of the variable anisotropic Hardy space $H^{p(\cdot)}_A({\mathbb {R}}^n)$, in this article, we establish the molecular decomposition of $H^{p(\cdot)}_A({\mathbb {R}}^n)$. Then, as applications, we further obtain
the boundedness of anisotropic Calder\'on-Zygmund operators
from $H^{p(\cdot)}_{A}(\rn)$ to $L^{p(\cdot)}(\rn)$ and from $H^{p(\cdot)}_{A}(\rn)$ to itself.

Precisely, this article is organized as follows.

In Section \ref{s2}, we first recall some notation and definitions
concerning expansive dilations, variable exponent, the variable Lebesgue space $L^{p(\cdot)}(\rn)$ and the variable anisotropic Hardy space $H^{p(\cdot)}_A({\mathbb {R}}^n)$, via the non-tangential grand maximal function. Then, motivated by Liu et al. \cite{lyy16,lyy17x} and Sawano et al. \cite{ns12}, we introduce the variable anisotropic molecular Hardy space $H^{p(\cdot),\,q,\,s,\,\varepsilon}_{A}(\rn)$
and establish its equivalence with $H^{p(\cdot)}_{A}(\rn)$ in Theorem \ref{t2.6}.
When it comes back to the isotropic setting, i.e., $A:=2{\rm I}_{n\times n}$,
the definition of the molecule in this article related to an integral size condition
covers the definition of molecule of Sawano et al. \cite[Definition 5.1]{ns12} related to a
 pointwise size condition. Thus, the molecular decomposition of $H^{p(\cdot)}_{2{\rm I}_{n\times n}}(\rn)$ in Theorem \ref{t2.6} also covers the molecular decomposition of $H^{p(\cdot)}(\rn)$ in \cite[Theorem 5.2]{ns12}, see Remarks \ref{r2.1x} and \ref{r2.7} for more details.

Section \ref{s3} is devoted to the proof of Theorem \ref{t2.6} via the
atomic characterization of $H^{p(\cdot)}_{A}(\rn)$ established in
\cite[Theorem 4.8]{lyy17x} (see also
Lemma \ref{l3.1} below). It is worth pointing
out that some of the proof methods of the molecular characterization of $H^{p}_{A}(\rn)=H^{p,\,p}_{A}(\rn)$ (\cite[Theorem 3.9]{lyy16}) and
$H^{p(\cdot)}(\rn)$ (\cite[Theorem 5.2]{ns12}) don't work anymore in the present setting.
For example, we search out some estimates related to $L^{p(\cdot)}(\rn)$ norms for some series of functions which can be reduced into dealing with the $L^q(\rn)$ norms of the corresponding functions (see Lemma \ref{l3.4} below). Then, by using this key lemma and the Fefferman-Stein vector-valued inequality of the Hardy-Littlewood maximal operator $M_{HL}$ on $L^{p(\cdot)}(\rn)$ (see Lemma \ref{l3.5} below), we prove
 their equivalences with $H^{p(\cdot)}_{A}(\rn)$ and $H^{p(\cdot),\,q,\,s,\,\varepsilon}_{A}(\rn)$.

In Section \ref{s4}, we first recall the definition of anisotropic
Calder\'on-Zygmund operators of Bownik \cite{b03}. Then, as an application of the molecular characterization of $H^{p(\cdot)}_{A}(\rn)$, we obtain the boundedness of anisotropic Calder\'on-Zygmund operators from $H^{p(\cdot)}_{A}(\rn)$ to $L^{p(\cdot)}(\rn)$ (see Theorem \ref{t4.1} below) or from $H^{p(\cdot)}_{A}$ to itself (see Theorem \ref{t4.2} below). Particularly, when $A:=2{\rm I}_{n\times n}$, we obtain the boundedness of the Riesz transforms (resp.,
when $n=1$, the Hilbert transform) from $H^{p(\cdot)}_{A}$ to $L^{p(\cdot)}$ or from $H^{p(\cdot)}_{A}$ to itself.

Finally, we make some conventions on notation.
Let $\nn:=\{1,\, 2,\,\ldots\}$ and $\zz_+:=\{0\}\cup\nn$.
For any $\az:=(\az_1,\ldots,\az_n)\in\zz_+^n:=(\zz_+)^n$, let
$|\az|:=\az_1+\cdots+\az_n$ and
$$\partial^\az:=
\lf(\frac{\partial}{\partial x_1}\r)^{\az_1}\cdots
\lf(\frac{\partial}{\partial x_n}\r)^{\az_n}.$$
Throughout the whole paper, we denote by $C$ a \emph{positive
constant} which is independent of the main parameters, but it may
vary from line to line.  For any $q\in[1,\,\infty]$, we denote by $q^{'}$ its conjugate index, namely, $1/q + 1/{q^{'}}=1$.
For any $a\in\rr$, $\lfloor a\rfloor$ denotes the
\emph{maximal integer} not larger than $a$.
The \emph{symbol} $D\ls F$ means that $D\le
CF$. If $D\ls F$ and $F\ls D$, we then write $D\sim F$.
If $E$ is a
subset of $\rn$, we denote by $\chi_E$ its \emph{characteristic
function}. If there are no special instructions, any space $\mathcal{X}(\rn)$ is denoted simply by $\mathcal{X}$. For instance, $L^2(\rn)$ is simply denoted by $L^2$.  Denote by $\cs$   the \emph{space of all Schwartz functions} and $\cs'$
its \emph{dual space} (namely, the \emph{space of all tempered distributions}).



\section{Molecular Characterization of $H^{p(\cdot)}_A$ \label{s2}}
\hskip\parindent
In this section, we first recall the notion of variable anisotropic Hardy space $H^{p(\cdot)}_{A}$, via the non-tangential grand maximal function $M_N(f)$, and then given its molecular characterization.

We begin with recalling the notion of {{expansive dilations}}
on $\rn$; see \cite[p.\,5]{b03}. A real $n\times n$ matrix $A$ is called an {\it
expansive dilation}, shortly a {\it dilation}, if
$\min_{\lz\in\sz(A)}|\lz|>1$, where $\sz(A)$ denotes the set of
all {\it eigenvalues} of $A$. Let $\lz_-$ and $\lz_+$ be two {\it positive numbers} such that
$$1<\lz_-<\min\{|\lz|:\ \lz\in\sz(A)\}\le\max\{|\lz|:\
\lz\in\sz(A)\}<\lz_+.$$ In the case when $A$ is diagonalizable over
$\mathbb C$, we can even take $\lz_-:=\min\{|\lz|:\ \lz\in\sz(A)\}$ and
$\lz_+:=\max\{|\lz|:\ \lz\in\sz(A)\}$.
Otherwise, we need to choose them sufficiently close to these
equalities according to what we need in our arguments.

It was proved in \cite[p.\,5, Lemma 2.2]{b03} that, for a given dilation $A$,
there exist a number $r\in(1,\,\fz)$ and a set $\Delta:=\{x\in\rn:\,|Px|<1\}$, where $P$ is some non-degenerate $n\times n$ matrix, such that $\Delta\subset r\Delta\subset A\Delta,$ and one can additionally
assume that $|\Delta|=1$, where $|\Delta|$ denotes the
$n$-dimensional Lebesgue measure of the set $\Delta$. Let
$B_k:=A^k\Delta$ for $k\in \zz.$ Then $B_k$ is open, $B_k\subset
rB_k\subset B_{k+1}$ and $|B_k|=b^k$, here and hereafter, $b:=|\det A|$.
An ellipsoid $x+B_k$ for some $x\in\rn$ and $k\in\zz$ is called a {\it dilated ball}.
Denote by $\mathfrak{B}$ the set of all such dilated balls, namely,
\begin{eqnarray}\label{e2.1}
\mathfrak{B}:=\{x+B_k:\ x\in \rn,\,k\in\zz\}.
\end{eqnarray}
Throughout the whole paper, let $\sigma$ be the {\it smallest integer} such that $2B_0\subset A^\sigma B_0$
and, for any subset $E$ of $\rn$, let $E^\complement:=\rn\setminus E$. Then,
for all $k,\,j\in\zz$ with $k\le j$, it holds true that
\begin{eqnarray}
&&B_k+B_j\subset B_{j+\sz},\label{e2.3}\\
&&B_k+(B_{k+\sz})^\complement\subset
(B_k)^\complement,\label{e2.4}
\end{eqnarray}
where $E+F$ denotes the {\it algebraic sum} $\{x+y:\ x\in E,\,y\in F\}$
of  sets $E,\, F\subset \rn$.

\begin{definition}
 A \textit{quasi-norm}, associated with
dilation $A$, is a Borel measurable mapping
$\rho_{A}:\rr^{n}\to [0,\infty)$, for simplicity, denoted by
$\rho$, satisfying
\begin{enumerate}
\item[\rm{(i)}] $\rho(x)>0$ for all $x \in \rn\setminus\{ \vec 0_n\}$,
here and hereafter, $\vec 0_n$ denotes the origin of $\rn$;

\item[\rm{(ii)}] $\rho(Ax)= b\rho(x)$ for all $x\in \rr^{n}$, where, as above, $b:=|\det A|$;

\item[\rm{(iii)}] $ \rho(x+y)\le H\lf[\rho(x)+\rho(y)\r]$ for
all $x,\, y\in \rr^{n}$, where $H\in[1,\,\fz)$ is a constant independent of $x$ and $y$.
\end{enumerate}
\end{definition}

In the standard dyadic case $A:=2{{\rm I}_{n\times n}}$, $\rho(x):=|x|^n$ for
all $x\in\rn$ is
an example of homogeneous quasi-norms associated with $A$, here and hereafter, ${\rm I}_{n\times n}$ denotes the $n\times n$ {\it unit matrix},
$|\cdot|$ always denotes the {\it Euclidean norm} in $\rn$.

It was proved, in \cite[p.\,6, Lemma 2.4]{b03},
that all homogeneous quasi-norms associated with a given dilation
$A$ are equivalent. Therefore, for a
given dilation $A$, in what follows, for simplicity, we
always use the {\it{step homogeneous quasi-norm}} $\rho$ defined by setting,  for all $x\in\rn$,
\begin{equation*}
\rho(x):=\sum_{k\in\zz}b^k\chi_{B_{k+1}\setminus B_k}(x)\ {\rm
if} \ x\ne \vec 0_n,\hs {\mathrm {or\ else}\hs } \rho(\vec 0_n):=0.
\end{equation*}
By \eqref{e2.3}, we know that, for all $x,\,y\in\rn$,
$$\rho(x+y)\le
b^\sz\lf(\max\lf\{\rho(x),\,\rho(y)\r\}\r)\le b^\sz[\rho(x)+\rho(y)];$$ see
\cite[p.\,8]{b03}. Moreover, $(\rn,\, \rho,\, dx)$ is a space of
homogeneous type in the sense of Coifman and Weiss \cite{cw71,cw77},
where $dx$ denotes the {\it $n$-dimensional Lebesgue measure}.

Now we recall that a measurable function $p(\cdot): \rn\rightarrow(0,\,\infty)$ is called a {\it variable exponent}. For any variable exponent $p(\cdot)$, let
\begin{eqnarray}\label{e2.5}
&&p_- :=\einf_{x\in\rn} p(x)\quad \mathrm{and} \ \ p_+ :=\esup_{x\in\rn} p(x).
\end{eqnarray}
Denote by $\cp$ the set of all variable exponents $p(\cdot)$ satisfying $0<p_-\leq p_+<\infty$.

Let $f$ be a measurable function on $\rn$ and $p(\cdot)\in\cp$. Then the {\it modular function} (or, for simplicity, the {\it modular}) $\varrho_{p(\cdot)}$, associated with $p(\cdot)$, is defined by setting
$$\varrho_{p(\cdot)}(f):=\int_\rn |f(x)|^{p(x)}\, dx$$
and the  {\it Luxemburg} (also called {\it Luxemburg-Nakano}) {\it quasi-norm} $\|f\|_{L^{p(\cdot)}}$ by
$$\|f\|_{L^{p(\cdot)}}:=\inf\lf\{\lambda \in(0,\,\infty):\varrho_{p(\cdot)}(f/\lambda)\leq 1\r\}.$$
Moreover, the {\it variable Lebesgue space} $L^{p(\cdot)}$ is defined to be the set of all measurable functions $f$
satisfying that $\varrho_{p(\cdot)}(f)<\infty$, equipped with the quasi-norm $\|f\|_{L^{p(\cdot)}}$.

The following remark comes from \cite[Remark 2.3]{lyy17}.
\begin{remark}\label{r2.1}
Let $p(\cdot)\in\cp$.
\begin{enumerate}
\item[\rm{(i)}]  Obviously, for any
$r\in(0,\,\infty)$ and $f\in L^{p(\cdot)}$,
$$\lf\|{|f|^r}\r\|_{L^{p(\cdot)}}=\lf\|f\r\|^r_{L^{rp(\cdot)}}.$$
Moreover, for any $\mu\in\ccc$ and $f,g\in L^{p(\cdot)}$,
$\lf\|\mu f\r\|_{L^{p(\cdot)}}=|\mu|\lf\|f\r\|_{L^{p(\cdot)}}$ and
$$\lf\|f+g\r\|^{\underline{p}}_{L^{p(\cdot)}}\leq\lf\|f\r\|^{\underline{p}}
_{L^{p(\cdot)}}+
\lf\|g\r\|^{\underline{p}}_{L^{p(\cdot)}},$$
here and hereafter,
\begin{align}\label{e2.5.1}
\underline{p}:=\min\{p_-,\,1\}
\end{align}
with $p_-$ as in \eqref{e2.5}. In particular, when $p_-\in[1,\,\infty]$, $L^{p(\cdot)}$ is a Banach space
(see \cite[Theorem 3.2.7]{dhh11}).
\item[\rm{(ii)}] It was proved in \cite[Proposition 2.21]{cf13} that, for any function $f\in L^{p(\cdot)}$
with $\lf\|f\r\|_{L^{p(\cdot)}}>0$,
$\varrho_{p(\cdot)}(f/{\|f\|_{L^{p(\cdot)}}})=1$
and, in \cite[Corollary 2.22]{cf13} that, if $\lf\|f\r\|_{L^{p(\cdot)}}\leq1$, then
$\varrho_{p(\cdot)}(f)\leq\lf\|f\r\|_{L^{p(\cdot)}}$.
\end{enumerate}
\end{remark}

A function $p(\cdot)\in \cp$ is said to satisfy the {\it globally log-H\"{o}lder continuous condition}, denoted by
$p(\cdot)\in C^{\log}$, if there exist two positive constants $C_{\log}(p)$ and $C_{\infty}$, and $p_{\infty}\in \rr$
such that, for any $x, y\in\rn$,
\begin{align*}
|p(x)-p(y)|\leq \frac{C_{\log}(p)}{\log(e+1/\rho(x-y))}
\end{align*}
and
\begin{align*}
|p(x)-p_{\infty}|\leq \frac{C_{\infty}}{\log(e+\rho(x))}.
\end{align*}

A $C^\infty$ function $\varphi$ is said to belong to the Schwartz class $\cs$ if,
for every integer $\ell\in\zz_+$ and multi-index $\alpha$,
$\|\varphi\|_{\alpha,\ell}:=\dsup_{x\in\rn}[\rho(x)]^\ell|\partial^\az\varphi(x)|
<\infty$.
The dual space of $\cs$, namely, the space of all tempered distributions on $\rn$ equipped with the weak-$\ast$
topology, is denoted by $\cs'$. For any $N\in\zz_+$, let
\begin{eqnarray*}
\cs_N:=\lf\{\varphi\in\cs:\ \|\varphi\|_{\alpha,\ell}\leq1,\ |\alpha|\leq N,\ \ \ell\leq N\r\}.
\end{eqnarray*}
In what follows, for $\varphi\in \cs$, $k\in\zz$ and $x\in\rn$, let $\varphi_k(x):= b^{-k}\varphi\lf(A^{-k}x\r)$.

\begin{definition}
Let $\varphi\in \cs$ and $f\in \cs'$. The{\it{ non-tangential maximal function}} $M_{\varphi}(f)$ with respect to $\varphi$ is defined by setting,
for any $x\in\rn$,
\begin{eqnarray*}
M_{\varphi}(f)(x):=\sup_{y\in x+B_k, k\in\zz}|f*\varphi_{k}(y)|.
\end{eqnarray*}
Moreover, for any given $N\in \nn$, the{\it{ non-tangential grand maximal function}} $M_{N}(f)$ of $f\in \cs'$ is defined by setting,
for any $x\in\rn$,
\begin{eqnarray*}
M_{N}(f)(x):=\sup_{\varphi\in \cs_N}M_{\varphi}(f)(x).
\end{eqnarray*}
\end{definition}

The following variable anisotropic Hardy space $H^{p(\cdot)}_{A}$ was introduced in \cite[Definition 2.4]{lyy17x}.
\begin{definition}\label{d2.4}
Let $p(\cdot)\in C^{\log}$, $A$ be a dilation and $N\in[\lfloor({1/\underline{p}}-1)/\ln\lambda_{-}\rfloor+2,\,\infty)$, where $\underline{p}$
is as in \eqref{e2.5.1}. The{\it{ variable anisotropic Hardy space}} $H_{A}^{p(\cdot)}$ is defined as
\begin{eqnarray*}
H_{A}^{p(\cdot)}:=\lf\{f\in \cs':M_{N}(f)\in L^{p(\cdot)}\r\}
\end{eqnarray*}
and, for any $f\in H_{A}^{p(\cdot)}$, let $\|f\|_{H_{A}^{p(\cdot)}}:=\|M_{N}(f)\|_{L^{p(\cdot)}}$.
\end{definition}
\begin{remark} Let $p(\cdot)\in C^{\log}$.
\begin{enumerate}
\item[\rm{(i)}]
The quasi-norm of $H_{A}^{p(\cdot)}$ in Definition \ref{d2.4} depends on $N$, however, by \cite[Theorem 3.10]{lyy17x}, we know that the $H_{A}^{p(\cdot)}$ is independent of the choice of $N$, as long as $N\in[\lfloor({1/\underline{p}}-1)/\ln\lambda_{-}\rfloor+2,\,\infty)$.
\item[\rm{(ii)}]  When
 $p(\cdot):=p$, where $p\in(0,\,\infty)$, the space $H^{p(\cdot)}_{A}$ is reduced to the anisotropic Hardy $H^{p}_{A}$ studied in \cite[Definition 3.11]{b03}.
\item[\rm{(iii)}]
When $A:=2{\rm I}_{n\times n}$,
the space $H^{p(\cdot)}_{A}$ is reduced to the variable Hardy space $H^{p(\cdot)}$ studied in \cite[p.\,3674]{ns12}.
\end{enumerate}
\end{remark}

The definition of variable isotropic molecules was introduced in \cite[Definition 5.1]{ns12}, and the definition of anisotropic
 molecules was introduced in \cite[Definition 3.7]{lyy16}. Motivated by the above two definitions, we introduce the definition of variable anisotropic molecules as follows.
\begin{definition}\label{d2.4x}
Let $p(\cdot)\in C^{\log}$, $q\in(1,\,\infty]$,
\begin{eqnarray}\label{e2.10}
s\in\lf[\lf\lfloor\lf(\frac{1}{p_-}-1\r)\frac{\ln b}{\ln\lambda_-}\r\rfloor,\,\infty\r)\cap\zz_+.
\end{eqnarray}
and $\varepsilon\in(0,\,\infty)$. A measurable function $m$ is called an
{\it{anisotropic $(p(\cdot),\,q,\,s,\,\ez)$-molecule}} associated with a dilated ball
$x_0+B_i\in\mathfrak{B}$ if
\begin{enumerate}
\item[\rm{(i)}]  for each
$j\in\zz_+$, $\|m\|_{L^q(U_j(x_0+B_i))}\le \frac{b^{-j\varepsilon}|B_{i}|^{1/q}}{\|\chi_{x_0+B_i}\|_{L^{p(\cdot)}}}$, where $U_0(x_0+B_i):=x_0+B_i$ and, for each $j\in\nn$, $U_j(x_0+B_{i}):=x_0+(A^jB_{i})\setminus(A^{j-1}B_{i})$;
\item[\rm{(ii)}] for all $\alpha\in \mathbb{Z}^n_+$ with $|\alpha|\leq s$, $\int_\rn m(x)x^\alpha dx=0$.
\end{enumerate}
\end{definition}
\begin{remark}\label{r2.1x}
Let $p(\cdot)\in C^{\log}$.
\begin{enumerate}
\item[(i)]
When the exponent function $p(\cdot)$ is reduced to the constant exponent $p$, i.e., $p(\cdot):=p\in(0,\,1]$,
the definition of
 the molecule in Definition \ref{d2.4x} is reduced to the molecule in \cite[Definition 3.7]{lyy16}.
\item[(ii)] When it comes back to the isotropic setting, i.e., $A:=2{\rm I}_{n\times n}$, and $\rho(x):= |x|^n$ for all $x\in\rn$, we claim that the definition of the molecule in \cite[Definition 5.1]{ns12} is also a
special form of the molecule in Definition \ref{d2.4}.
Recall that, in Sawano et al.\cite[Definition 5.1]{ns12}, it is obvious that the definition of molecule associated with a cube which can be equivalently replaced by a Euclidean ball
as follows.
Let $0< p_- \leq p_+ <q \leq \infty$, $q\geq1$ and $d\in\zz\cap[d_{p(\cdot)},\,\infty)$ with $d_{p(\cdot)}=\min \{d\in \nn \cup {0}: p_-(n+d+1)>n\}$ be fixed.
 A function $M$
is a $(p(\cdot),\,q)$-molecule centered at the ball $B=:B(x_0,\,r)\subset\rn$, centered at $x_0$ with radius $r$, if it satisfies the following conditions:
\begin{enumerate}
\item[\rm{(1)}]$M$ satisfies the estimate $$\|M\|_{L^q(B)}\leq\frac{|B|^{1/q}}{\|\chi_B\|_{L^{p(\cdot)}}};$$
\item[\rm{(2)}]For any $x\in B^{\complement}$, $M(x)$ satisfies the estimate  $$|M(x)|\leq\frac{\lf(1+\frac{|x-x_0|}{r}\r)^{-2n-2d-3}}
    {\|\chi_B\|_{L^{p(\cdot)}}};$$

\item[\rm{(3)}]If $\alpha$ is a multi-index with length less than or equal to $d$, then we have
$$\int_\rn x^\alpha M(x)\,dx=0.$$
\end{enumerate}

 To prove our claim, let $M$ be a $(p(\cdot),\,q)$-molecule as above. We only need to show that $M$ satisfies the above (ii) of Definition \ref{d2.4}, for each
$j\in\zz_+$,
\begin{align*}
\|M\|_{L^q(U_j(x_0,\,r))}&=\lf(\int_{U_j({B(x_0,\,r))}}|M(x)|^q\,dx\r)^{1/q}\\
&\leq\lf[\int_{U_j({B(x_0,\,r))}}\lf(1+\frac{|x-x_0|}{r}\r)^{q(-2n-2d-3)}\,dx
\r]^{1/q}\lf\|\chi_{B(x_0,\,r)}\r\|_{L^{p(\cdot)}}^{-1}\\
&\lesssim(1+2^j)^{-2n-2d-3}(2^jr)^{n/q}\lf\|\chi_{B(x_0,\,r)}\r\|_{L^{p(\cdot)
}}^{-1}\\
&\lesssim(2^j)^{-2n-2d-3}(2^j)^{n/q}\frac{|B(x_0,\,r)|^{1/q}}
{\lf\|\chi_{B(x_0,\,r)}\r\|_{L^{p(\cdot)}}}\\
&\thicksim2^{-j(2n+2d+3-n/q)}\frac{|B(x_0,\,r)|^{1/q}}
{\lf\|\chi_{B(x_0,\,r)}\r\|_{L^{p(\cdot)}}}\\
&\thicksim2^{-j\varepsilon}\frac{|B(x_0,\,r)|^{1/q}}
{\lf\|\chi_{B(x_0,\,r)}\r\|_{L^{p(\cdot)}}},
\end{align*}
where $\varepsilon:=2n+2d+3-{n}/{q}>0$. Thus, the above claim holds true.
\end{enumerate}
\end{remark}

In what follows, we call an anisotropic $(p(\cdot),\,q,\,s,\,\varepsilon)$-molecule simply by
$(p(\cdot),\,q,\,s,\,\varepsilon)$-molecule. Via $(p(\cdot),\,q,\,s,\,\varepsilon)$-molecules,
we introduce the following variable anisotropic  molecular Hardy space
$H^{p(\cdot),\,q,\,s,\,\varepsilon}_{A}$.

\begin{definition}\label{d2.5}
Let $p(\cdot)\in C^{\log}$, $q\in(1,\,\infty]$, $A$ be a dilation and
let $s$ be as in \eqref{e2.10}.
 The {\it variable anisotropic molecular Hardy space}
$H^{p(\cdot),\,q,\,s,\,\varepsilon}_{A}$
is defined to be the set of all distributions $f\in \cs'$ satisfying that there
exist $\{\lambda_i\}_{i\in\nn}\subset\ccc$ and a sequence of
$(p(\cdot),\,q,\,s,\,\varepsilon)$-molecules, $\{m_i\}_{i\in\nn}$, associated, respectively,
with $\{{B^{(i)}}\}_{i\in\nn}\subset\mathfrak{B}$ such that
\begin{align*}
f=\sum_{i\in\nn} \lz_{i}m_i \ \ \mathrm{in\ } \ \cs'.
\end{align*}
Moreover, for any $f\in H^{p(\cdot),\,q,\,s,\,\varepsilon}_{A}$, let
$$\|f\|_{H^{p(\cdot),\,q,\,s,\,\varepsilon}_{A}}
:=\inf \lf\|\lf\{\sum_{i\in\nn} \lf[\frac{|\lambda_i|\chi_{{B^{(i)}}}}{\|\chi_{{B^{(i)}}}
\|_{L^{p(\cdot)}}}\r]^{\underline{p}}\r\}^{1/\underline{p}}\r\|
_{L^{p(\cdot)}},$$
where the infimum is taken over all the decompositions of $f$ as above.
\end{definition}

The following Theorem \ref{t2.6} shows the molecular characterization of
$H^{p(\cdot)}_{A}$, whose proof is given in the next section.
\begin{theorem}\label{t2.6}
Let $p(\cdot)\in C^{\log}$ and $q\in(1,\,\infty]\cap(p_+,\,\infty]$ with $p_+$ as in \eqref{e2.5},
$s$ be as in \eqref{e2.10}, $\varepsilon\in (\max\{1,\,(s+1)\log_b(\lambda_+)\},\,\infty)$ and $N\in\nn\cap[\lfloor(1/{\underline{p}}-1) {\ln b/\ln \lambda_-}\rfloor+2,\,\infty)$
with $\underline{p}$ as in \eqref{e2.5.1}.
Then $$H^{p(\cdot)}_{A}=H^{p(\cdot),\,q,\,s,\,\varepsilon}_{A}$$ with equivalent quasi-norms.
\end{theorem}
\begin{remark}\label{r2.7}
Let $p(\cdot)\in C^{\log}$.
\begin{enumerate}
\item[(i)]Liu et al. \cite{lyy17} introduced the anisotropic variable Hardy-Lorentz space $H^{p(\cdot),\,q}
_A$, where $p(\cdot) : {\mathbb {R}}^n\rightarrow
(0,\infty]$ is a variable exponent function satisfying the so-called globally log-H\"{o}lder continuous condition and $q\in(0,\infty]$.
When the exponent function $p(\cdot)$ is reduced to the constant exponent $p$, i.e., $p(\cdot):=p\in(0,\,1]$, the space $H^{ p(\cdot),\,q}_
A$ is just the
anisotropic Hardy-Lorentz space $H^{p,\,q}_A$ in \cite{lyy16}. When $q=p$,  the
molecular characterization of $H^{p(\cdot)}_{A}$ in Theorem \ref{t2.6} is reduced to the molecular characterization of anisotropic Hardy spaces $H^{p}_{A}$=$H^{p,\,p}_{A}$ in \cite[Theorem 3.9]{lyy16} and also covers the molecular characterization of
$H^\vz_A({\mathbb {R}}^n)$ with $\vz(x,\,t):=t^p$ in \cite[Theorem 2.10]{lffy16}, where the molecule is defined associated with a pointwise size condition.
\item[(ii)] When it comes back to the isotropic setting, i.e., $A:=2{\rm I}_{n\times n}$, the
 molecular characterization of $H^{p(\cdot)}_{A}$ in Theorem \ref{t2.6} is still new and covers the
 molecular characterization of $H^{p(\cdot)}$ in \cite[Theorem 5.2]{ns12}.
\end{enumerate}
\end{remark}


\section{Proof of Theorem \ref{t2.6}\label{s3}}
\hskip\parindent
To show Theorem \ref{t2.6},
we begin with
the following notion of anisotropic $(p(\cdot),\,q,\,s)$-atoms introduced in \cite[Definition 4.1]{lyy17}.

\begin{definition}\label{d3.1}
Let $p(\cdot)\in\cp$, $q\in(1,\,\infty]$ and
$s\in[\lfloor(1/{p_-}-1) {\ln b/\ln \lambda_-}\rfloor,\,\infty)\cap\zz_+$ with $p_-$ as in \eqref{e2.5}. An {\it anisotropic $(p(\cdot),\,q,\,s)$-atom} is a measurable function $a$ on $\rn$ satisfying
\begin{enumerate}
\item[\rm{(i)}]  $\supp a\subset B$, where $B\in\mathfrak{B}$ and $\mathfrak{B}$ is as in \eqref{e2.1};
\item[\rm{(ii)}] $\|a\|_{L^q}\le \frac{|B|^{1/q}}{\|\chi_B\|_{L^{p(\cdot)}}}$;
\item[\rm{(iii)}] $\int_\rn a(x)x^\alpha dx=0$ for any $\alpha\in \mathbb{Z}^n_+$ with $|\alpha|\leq s$.
\end{enumerate}
\end{definition}
Throughout this article, we call an anisotropic $(p(\cdot),\,q,\,s)$-atom simply by
a $(p(\cdot),\,q,\,s)$-atom. The following variable anisotropic atomic Hardy space
was introduced in \cite[Definition 4.2]{lyy17x}
\begin{definition}
Let $p(\cdot)\in C^{\log}$, $q\in(1,\,\infty]$, $A$ be a dilation and
$s$ be as in \eqref{e2.10}.
 The {\it variable anisotropic atomic Hardy space}
$H^{p(\cdot),\,q,\,s}_{A}$
is defined to be the set of all distributions $f\in \cs'$ satisfying that there
exist $\{\lambda_i\}_{i\in\nn}\subset\ccc$ and a sequence of
$(p(\cdot),\,q,\,s)$-atoms, $\{a_i\}_{i\in\nn}$, supported, respectively,
on $\{{B^{(i)}}\}_{i\in\nn}\subset\mathfrak{B}$ such that
\begin{align*}
f=\sum_{i\in\nn} \lz_{i}a_i \ \ \mathrm{in\ } \ \cs'.
\end{align*}
Moreover, for any $f\in H^{p(\cdot),\,q,\,s}_{A}$, let
$$\|f\|_{H^{p(\cdot),\,q,\,s}_{A}}
:=\inf \lf\|\lf\{\sum_{i\in\nn} \lf[\frac{|\lambda_i|\chi_{{B^{(i)}}}}{\|\chi_{{B^{(i)}}}
\|_{L^{p(\cdot)}}}\r]^{\underline{p}}\r\}^{1/\underline{p}}\r\|
_{L^{p(\cdot)}},$$
where the infimum is taken over all the decompositions of $f$ as above.
\end{definition}

The following lemma reveals the atomic characterization of the variable anisotropic Hardy space
(see \cite[Theorem 4.8]{lyy17x}).

\begin{lemma}\label{l3.1}
Let $p(\cdot)\in C^{\log}$, $q\in(\max\{p_+,\,1\},\,\infty]$
with $p_+$ as in \eqref{e2.5},
$s\in[\lfloor(1/{p_-}-1) {\ln b/\ln \lambda_-}\rfloor,\,\infty)\cap\zz_+$ with $p_-$ as in \eqref{e2.5} and
$N\in\nn\cap[\lfloor(1/{\underline{p}}-1) {\ln b/\ln \lambda_-}\rfloor+2,\,\infty)$.
Then $$H^{p(\cdot)}_{A}=H^{p(\cdot),\,q,\,s}_{A}$$ with equivalent quasi-norms.
\end{lemma}

The following lemma is just \cite[Lemma 4.5]{lyy17}.
\begin{lemma}\label{l3.4}
Let $p(\cdot)\in C^{\log}$, $t\in(0,\,\underline{p}]$ with $\underline{p}$ as in \eqref{e2.5.1} and $r\in[1,\,\infty]\cap(p_+,\,\infty]$ with $p_+$ as in \eqref{e2.5}.
Then there exists a positive constant $C$ such that, for any sequence $\{B^{(k)}\}_{k\in\nn}\subset\mathfrak{B}$ of dilated
balls, numbers $\{\lambda_k\}_{k\in\nn}\subset\ccc$ and measurable functions $\{a_k\}_{k\in\nn}$ satisfying that, for each
$k\in\nn$, $\supp a_k\subset B^{(k)}$ and $\|a_k\|_{L^r}\leq|B^{(k)}|^{1/r}$, it holds true that
$$\lf\|\lf(\sum_{k\in\nn}\lf|\lambda_ka_k\r|^t\r)^{1/t}\r\|_{L^{p(\cdot)}}\leq C
\lf\|\lf(\sum_{k\in\nn}\lf|\lambda_k\chi_{B^{(k)}}\r|^t\r)^{1/t}
\r\|_{L^{p(\cdot)}}.$$
\end{lemma}

We recall the definition of anisotropic Hardy-Littlewood maximal function $M_{HL}(f)$. For any $f\in L_{loc}^1$ and $x\in \rn$,
\begin{align}\label{e2.9}
M_{HL}(f)(x):=\sup_{x\in B\in\mathfrak{B}}\frac{1}{|B|}\int_{B}|f(z)|\,dz,
\end{align}
where $\mathfrak{B}$ is as in \eqref{e2.1}.

The following lemma is just \cite[Lemma 4.3]{lyy17}.
\begin{lemma}\label{l3.5}
Let $q\in(1,\,\infty]$. Assume that $p(\cdot)\in C^{\log}$ satisfies $1<p_-\leq p_+ <\infty$, where $p_-$ and $p_+$ are as in \eqref{e2.5}.
Then there exists a positive constant $C$ such that, for any sequence $\{f_k\}_{k\in\nn}$ of measurable functions,
$$\lf\|\lf\{\sum_{k\in\nn}\lf[M_{HL}(f_k)\r]^q\r\}^{1/q}\r\|_{L^{p(\cdot)}}\leq C
\lf\|\lf(\sum_{k\in\nn}|f_k|^q\r)^{1/q}\r\|_{L^{p(\cdot)}}$$
with the usual modification made when $q=\infty$, where $M_{HL}$ denotes the Hardy-Littlewood maximal operator as in \eqref{e2.9}.
\end{lemma}

The proof of the following Theorem \ref{t2.6} is motivated by \cite[Theorem 4.8]{lyy17x}.
\begin{proof}[Proof of Theorem \ref{t2.6}]
By the definitions of $(p(\cdot),\,q,\,s)$-atom and $(p(\cdot),\,q,\,s,\,\varepsilon)$-molecule,
we notice that a $(p(\cdot),\,\infty,\,s)$-atom is also a $(p(\cdot),\,q,\,s,\,\varepsilon)$-molecule, which implies that
$$H^{p(\cdot),\,\infty, \,s}_{A}\subset H^{p(\cdot),\,q,\,s,\,\varepsilon}_{A}.$$
This, combined with Lemma \ref{l3.1}, further implies that, to prove Theorem \ref{t2.6}, it suffices to show
$H^{p(\cdot),\,q,\,s,\,\varepsilon}_{A}\subset H^{p(\cdot)}_{A}$.

To show this, for any $f\in H^{p(\cdot),\,q,\,s,\,\varepsilon}_{A}$, by Definition \ref{d2.5},
we deduce that there exists a sequence of $(p(\cdot),\,q,\,s,\,\varepsilon)$-molecules,
$\{m_i\}_{i\in\nn}$, associated with dilated balls $\{B^{(i)}\}_{i\in\nn}\subset\mathfrak{B}$,
where $B^{(i)}:=x_i+B_{\ell_i}$ with $x_i\in\rn$ and $\ell_i\in\zz$, such that
 $$f=\sum_{i\in\nn} \lz_im_i\,\,\mathrm{in} \,\,\mathcal{S}^{'}$$
and
\begin{align}\label{e3.3}
\|f\|_{H^{p(\cdot),\,q,\,s,\,\ez}_{A}}
\thicksim \lf\|\lf\{\sum_{i\in\nn} \lf[\frac{|\lambda_i|\chi_{{B^{(i)}}}}{\|\chi_{{B^{(i)}}}\|_{L^{p(\cdot)}}}\r]^{\underline{p}}\r\}^{1/\underline{p}}\r\|
_{L^{p(\cdot)}}.
\end{align}

To prove $f\in H^{p(\cdot)}_{A}$,
it is easy to see that, for all $N\in\mathbb{N}\cap[\lfloor({1/\underline{p}}-1)/\ln\lambda_{-}\rfloor+2,\,\infty)$,
\begin{align*}
\|M_{N}(f)\|_{L^{p(\cdot)}}^{\underline{p}}&=\lf\|M_{N}\lf(\sum_{i\in \nn}\lambda_{i}m_{i}\r)\r\|_{L^{p(\cdot)}}^{\underline{p}}\\
&\leq\lf\|\sum_{i\in \nn}|\lambda_{i}|M_{N}(m_{i})\r\|_{L^{p(\cdot)}}^{\underline{p}}\\
&\leq\lf\|\sum_{i\in \nn}|\lambda_{i}|M_{N}(m_{i})\chi_{A^{\sigma}B^{(i)}}
\r\|_{L^{p(\cdot)}}^{\underline{p}}+\lf\|\sum_{i\in \nn}|\lambda_{i}|M_{N}(m_{i})\chi_{({A^{\sigma}B^{(i)}}
)^{\complement}}\r\|_{L^{p(\cdot)}}^{\underline{p}}\\
&\leq\lf\|\lf\{\sum_{i\in \nn}[|\lambda_{i}|M_{N}(m_{i})\chi_{A^{\sigma}B^{(i)}}]^{\underline{p}}\r\}^{1/\underline{p}}\r\|_{L^{p(\cdot)}}^{\underline{p}}+\lf\|\sum_{i\in \nn}|\lambda_{i}|M_{N}(m_{i})\chi_{({A^{\sigma}B^{(i)}})^{\complement}}\r\|_{L^{p(\cdot)}}^{\underline{p}}\\
&=:\mathrm{I_{1}}+\mathrm{I_{2}},
\end{align*}
where $A^{\sigma}B^{(i)}$ is the $A^{\sigma}$ concentric expanse on $B^{(i)}$, that is, $A^{\sigma}B^{(i)}:=x_i+A^{\sigma}B_{\ell_i}$.

To estimate $\mathrm{I_{1}}$, for any $\widetilde{q}\in((\max\{p_+,\,1\},\,q)$, by
 the boundedness of $M_{N}$ on $L^{r}$ for all $r\in(1,\,\infty)$
and H\"{o}lder's inequality,
we have
\begin{align}\label{e3.9x}
&\lf\|\lf\|\chi_{B^{(i)}}\r\|_{L^{p(\cdot)}} M_{N}(m_i)\chi_{A^{\sigma}{B^{(i)}}}\r\|_{L^{\tilde{q}}}
\\
\leq&\lf\|\chi_{B^{(i)}}\r\|_{L^{p(\cdot)}}\lf\|M_{N}(m_i)\r\|_{L^{\tilde{q}}}\nonumber \\
\lesssim&\lf\|\chi_{B^{(i)}}\r\|_{L^{p(\cdot)}}\lf\|m_i\r\|_{L^{\tilde{q}}}\nonumber \\
\thicksim&\lf\|\chi_{B^{(i)}}\r\|_{L^{p(\cdot)}}\sup_{\|g\|_{L^{{\tilde{q}}'}}=1}\lf|\int_\rn m_i(x)g(x)\,dx\r|\nonumber\\
\lesssim&
\lf\|\chi_{B^{(i)}}\r\|_{L^{p(\cdot)}}\sup_{\|g\|_{L^{{\tilde{q}}'}}=1} \sum_{j\in\zz_+}\int_{U_j(B^{(i)})} \lf|m_i(x)\r||g(x)|\,dx\nonumber\\
\lesssim&
\lf\|\chi_{B^{(i)}}\r\|_{L^{p(\cdot)}}\sup_{\|g\|_{L^{{\tilde{q}}'}}=1}\sum_{j\in\zz_+}\lf[\int_{U_j(B^{(i)})} \lf|m_i(x)\r|^{q}dx\r]^{1/q}
\lf[\int_{U_j(B^{(i)})} |g(x)|^{q'}dx\r]^{1/q'}\nonumber\\
\thicksim&
\lf\|\chi_{B^{(i)}}\r\|_{L^{p(\cdot)}}\sup_{\|g\|_{L^{{\tilde{q}}'}}=1}\sum_{j\in\zz_+}\lf\|m_i\r\|_{L^{q}(U_j(B^{(i)}))}
\lf[\int_{U_j(B^{(i)})} |g(x)|^{q'}dx\r]^{1/q'},\nonumber
\end{align}
where, for any $i\in\nn$, $U_0(B^{(i)}):=B^{(i)}$
and, for any $j\in\mathbb{N}$,
\begin{eqnarray}\label{e3.10}
U_j(B^{(i)}):=x_i+(A^jB_{\ell_i})\backslash (A^{j-1}B_{\ell_i}).
\end{eqnarray}
Moreover, by \eqref{e3.10} and H\"{o}lder's inequality, we know that, for any $\widetilde{q}\in((\max\{p_+,\,1\},\,q)$, $i\in\nn$ and $j\in\mathbb{Z}_+$,
\begin{align*}
\lf[\int_{U_j(B^{(i)})} |g(x)|^{q'}dx\r]^{1/q'}
\leq&\lf|A^jB_{\ell_i}\r|^{1/q'}\lf[\frac{1}{\lf|A^jB_{\ell_i}\r|}\int_{x_i+A^jB_{\ell_i}} |g(x)|^{q'}dx\r]^{1/q'}\\
\hs\leq&\lf|A^jB_{\ell_i}\r|^{1/q'}\inf_{x\in{x_i+B_{\ell_i}}}
\lf[M_{HL}\lf(|g|^{q'}\r)(x)\r]^{1/q'}\\
\hs\leq& \lf|A^jB_{\ell_i}\r|^{1/q'}\lf\{\frac{1}{\lf|B_{\ell_i}\r|}\int_{x_i+B_{\ell_i}} \lf[M_{HL}\lf(|g|^{q'}\r)(x)\r]^{\tilde{q}'/q'}dx\r\}^{1/\tilde{q}'}.
\end{align*}
Substituting the above inequality into \eqref{e3.9x}, by $\lf\|m_i\r\|_{L^{q}({U_j(B^{(i)})})}\le \frac{b^{-j\varepsilon}|B^{(i)}|^{1/{q}}}{\|\chi_{B^{(i)}}\|_{L^{p(\cdot)}}}$,
 $1/{q'}<1<\varepsilon$
and applying the fact that $M_{HL}$ is bounded on $L^r$ for all $r\in(1,\,\infty)$,
we conclude that, for any $\widetilde{q}\in((\max\{p_+,\,1\},\,q)$ and $i\in\nn$,
\begin{align*}
&\lf\|\lf\|\chi_{B^{(i)}}\r\|_{L^{p(\cdot)}} M_{N}(m_i)\chi_{A^{\sigma}{B^{(i)}}}\r\|_{L^{\tilde{q}}}
\nonumber\\
\hs\leq&\lf\|\chi_{B^{(i)}}\r\|_{L^{p(\cdot)}}
\sup_{\|g\|_{L^{\tilde{q}'}}=1}\sum_{j\in\zz_+}b^{j(1/{q'}-\varepsilon)}
\frac{|B^{(i)}|^{1/\tilde{q}}}{\|\chi_{B^{(i)}}\|_{L^{p(\cdot)}}}\\
\hs\hs&\times
\lf\{\int_{x_i+B_{\ell_i}} \lf[M_{HL}\lf(|g|^{q'}\r)(x)\r]^{\tilde{q}'/q'}dx\r\}^{1/\tilde{q}'}\\
\hs\ls&\lf|B^{(i)}\r|^{{1}/{\tilde{q}}}\sup_{\|g\|_{L^{\tilde{q}'}}=1}
\lf\{\int_{x_i+B_{\ell_i}} \lf[M_{HL}\lf(|g|^{q'}\r)(x)\r]^{\tilde{q}'/q'}dx\r\}^{1/\tilde{q}'}\\
\hs\ls&\lf|B^{(i)}\r|^{{1}/{\tilde{q}}}\sup_{\|g\|_{L^{\tilde{q}'}}=1}
\lf[\int_{\rn} |g(x)|^{\tilde{q}'}dx\r]^{1/\tilde{q}'}\\
\hs\sim&\lf|B^{(i)}\r|^{1/{\tilde{q}}},
\end{align*}
which, combined with Lemma \ref{l3.4}, $\widetilde{q}\in((\max\{p_+,\,1\},\,q)$ and \eqref{e3.3}, we obtain
\begin{align*}
\mathrm{I_{1}}
&=\lf\|\lf\{\sum_{i\in\nn}\lf[\frac{|\lambda_i|}{\|\chi_{{B^{(i)}}}\|_{L^{p(\cdot)}}}\lf\|\chi_{B^{(i)}}\r\|_{L^{p(\cdot)}}M_{N}(m_{i})\chi_{A^{\sigma}{B^{(i)}}}\r]^{\underline{p}}\r\}^
{1/\underline{p}}\r\|_{L^{p(\cdot)}}^{\underline{p}}\\
&\lesssim\lf\|\lf\{\sum_{i\in\nn}\lf[\frac{|\lambda_i|}{\|\chi_{{B^{(i)}}}\|_{L^{p(\cdot)}}}\chi_{{B^{(i)}}}\r]^{\underline{p}}\r\}^
{1/\underline{p}}\r\|_{L^{p(\cdot)}}^{\underline{p}}\\
&\thicksim\|f\|_{H^{p(\cdot),\,q,\,s,\,\ez}_{A}}^{\underline{p}}.
\end{align*}

To deal with ${\rm{I_{2}}}$, for any $i\in\nn$ and $x\in(x_i+A^\sigma B_{\ell_i}))^\complement$, by $M_{N}(f)(x)\thicksim M_{N}^0(f)(x)$ \cite[Proposition 3.10]{b03} and proceeding as in the proof of \cite[(3.48)]{lyy16}, we know that
\begin{eqnarray}\label{e3.12}
M_{N}(m_i)(x)\ls\lf\|\chi_{B^{(i)}}\r\|^{-1}_{L^{p(\cdot)}}
\frac{\lf|{B^{(i)}}\r|^\beta}{[\rho(x-x_i)]^\beta}
\ls\lf\|\chi_{B^{(i)}}\r\|^{-1}_{L^{p(\cdot)}}
\lf[M_{HL}(\chi_{B^{(i)}})(x)\r]^\beta,
\end{eqnarray}
where, for any $i\in\nn$, $x_i$ denotes the centre of the dilated ball $B^{(i)}$ and
\begin{eqnarray}\label{e3.12.1}
\beta:=\lf(\frac{\ln b}{\ln \lambda_-}+s+1\r)\frac{\ln \lambda_-}{\ln b}>\frac{1}{\underline{p}}.
\end{eqnarray}
From this, Remark \ref{r2.1}(i), Lemma \ref{l3.5} and \eqref{e3.3}, we deduce that
\begin{align}\label{e3.10x}
\mathrm{I}_{2}&\lesssim\lf\|\sum_{i\in\nn}\frac{|\lambda_i|}{\|\chi_{B^{(i)}}\|_{L^{p(\cdot)}}}[M_{HL}(\chi_{B^{(i)}})]^{\beta}\r\|_{L^{p(\cdot)}}^{\underline{p}}\\
&\thicksim\lf\|\lf\{\sum_{i\in\nn}\frac{|\lambda_i|}{\|\chi_{B^{(i)}}\|_{L^{p(\cdot)}}}[M_{HL}(\chi_{B^{(i)}})]^{\beta}\r\}^{1/\beta}\r\|_{L^{\beta p(\cdot)}}^{{\beta}{\underline{p}}}\nonumber\\
&\lesssim\lf\|\lf\{\sum_{i\in\nn}\frac{|\lambda_i|\chi_{B^{(i)}}}{\|\chi_{B^{(i)}}\|_{L^{p(\cdot)}}}\r\}^{1/\beta}\r\|_{L^{\beta p(\cdot)}}^{{\beta}{\underline{p}}}\nonumber
\thicksim\lf\|\sum_{i\in\nn}\frac{|\lambda_i|\chi_{B^{(i)}}}{\|\chi_{B^{(i)}}\|_{L^{p(\cdot)}}}\r\|_{L^{p(\cdot)}}^{\underline{p}}\nonumber\\
&\lesssim\lf\|\lf\{\sum_{i\in\nn}\lf[\frac{|\lambda_i|\chi_{B^{(i)}}}{\|\chi_{B^{(i)}}\|_{L^{p(\cdot)}}}\r]^{\underline{p}}\r\}^{1/\underline{p}}\r\|_{L^{p(\cdot)}}^{\underline{p}}\nonumber
\thicksim\|f\|_{H^{p(\cdot),\,q,\,s,\,\ez}_{A}}^{\underline{p}}.\nonumber
\end{align}
This, together with $\mathrm{I_1}$ and $\mathrm{I_2}$, shows that
\begin{align*}
\|f\|_{H_{A}^{p(\cdot)}}\thicksim\|M_{N}(f)\|_{L^{p(\cdot)}}\lesssim\|f\|_{H^{p(\cdot),\,q,\,s,\,\ez}_{A}}.
\end{align*}
This implies that $f\in H^{p(\cdot)}_{A}$ and hence
$H^{p(\cdot),\,q,\,s,\,\varepsilon}_{A}\subset H^{p(\cdot)}_{A}$. This finishes the proof of Theorem \ref{t2.6}.
\end{proof}

\section{Applications}\label{s4}
\hskip\parindent

In this section, as an application of the molecular characterization of $H^{p(\cdot)}_{A}$ in
Theorem \ref{t2.6}, we obtain the boundedness of anisotropic Calder\'on-Zygmund
operators from $H^{p(\cdot)}_{A}$ to $L^{p(\cdot)}$ or from $H^{p(\cdot)}_{A}$ to itself.
Particularly, when $A:=2\rm{I_{n\times n}}$, we obtain the boundedness of the Riesz transforms (resp.,
when $n=1$, the Hilbert transform) from $H^{p(\cdot)}_{A}$ to $L^{p(\cdot)}$ or from $H^{p(\cdot)}_{A}$ to itself.

Let us begin with the notion of anisotropic Calder\'on-Zygmund operators
associated with dilation $A$.
\begin{definition}\label{d4.1}
A locally integrable function $K$ on $\Omega:=\{(x,\,y)\in\rn\times\rn:\,x\ne y\}$ is called
an {\it anisotropic Calder\'on-Zygmund kernel}
(with respect to a dilation A and a quasi-norm $\rho$) if there exists a positive constant $C$ such that
$$\int_{(y+B_{l+2\sigma})^{\complement}}\lf|K(x,\,y)-K(x,\,\wz y)\r|\,dx\le C\,\,\,\,\mathrm{whenever}\,\,\, \wz y\in y+B_l, $$
for any $y\in\rn$ and $l\in\zz$.

We call that
$T$ is an {\it anisotropic Calder\'on-Zygmund operator} if $T$ is a continuous linear operator
mapping $\mathcal{S}$ into $\mathcal{S}'$ that extends to a bounded linear
operator on $L^2$ and there exists an anisotropic Calder\'on-Zygmund kernel $K$ such that,
for all $f\in C_c^\infty$ and $x\not\in \supp (f)$,
$$Tf(x):=\int_\rn K(x,\,y)f(y)dy.$$
\end{definition}
\begin{theorem}\label{t4.0}
Let $p\in(1,\,\infty)$ and $T$ is an anisotropic Calder\'on-Zygmund operator. Then $T$ extends to a bounded linear operator on $L^p$.
\end{theorem}
\begin{proof}
By \cite[p.\,22, Corrollary]{s93}, we see that anisotropic Calder\'on-Zygmund operator $T$ can be extended to a bounded linear operator on
 $L^p$ with $p\in(1,\,2]$.
By taking duals we obtain that $T$ is also bounded from $L^p$ to $L^p$ for $p\in[2,\,\infty)$. From the above, we conclude that anisotropic Calder\'on-Zygmund operator $T$ is bounded on $L^p$ with $p\in(1,\,\infty)$.
\end{proof}
To obtain the boundedness of anisotropic Calder\'on-Zygmund
operators on $H^{p(\cdot)}_{A}$, we need to increase the smooth hypothesis
on the kernel $K$. The following definition was introduced by Bownik in \cite[Definition 9.2]{b03}.

\begin{definition}\label{d4.3}
Let $N\in\zz_+$. We say that $T$ is an {\it anisotropic Calder\'on-Zygmund operator of order $N$} if $T$ satisfies Definition \ref{d4.1} with the kernel $K$ in the class $C^{N}$ as a function of $y$.
We also require that there exists a positive constant $C$ such that, for any $\alpha\in\zz^n_+$, with $|\alpha|\le N$, and $(x,y)\in\Omega$,
\begin{eqnarray}\label{e4.1}
\lf|\partial^\alpha_y \lf[K\lf(\cdot,\,A^\ell \cdot\r)\r]\lf(x,\,A^{-\ell}y\r)\r|\le C[\rho(x-y)]^{-1}=Cb^{-\ell},
\end{eqnarray}
where $\ell\in\zz$ is the unique integer such that $\rho(x-y)=b^\ell$ with
the implicit equivalent positive constants independent of $x,\,y$ and $\ell$.
More formally,
$$\partial^\alpha_y \lf[K\lf(\cdot,\,A^\ell \cdot\r)\r]\lf(x,\,A^{-\ell}y\r)$$
means
$(\partial^\alpha_y \tilde{K})(x,\,A^{-\ell}y),$ where $\tilde{K}(x,\,y):=K(x,\,A^\ell y)$ for all $(x,\,y)\in\rn$ and $x\ne A^\ell y$.
\end{definition}

\begin{remark}
In Definition \ref{d4.3}, when $N\in\zz_+$,  $A:=2{\mathrm I_{n\times n}}$
and $\rho(x):=|x|^n$ for all $x\in\rn$, then \eqref{e4.1} becomes that,
for any $\alpha\in\zz^n_+$, with $|\alpha|\le N$, and $(x,y)\in\Omega$,
\begin{eqnarray}\label{e4.2}
|\partial^\az K(x,\,y)|\le C|x-y|^{-n-|\az|},
\end{eqnarray}
which is standard and well known. More examples of anisotropic Calder\'on-Zygmund operator of order $N$ as in Definition \ref{d4.1}; see \cite[p.\,61]{b03}.
\end{remark}

To obtain the boundedness of anisotropic Calder\'on-Zygmund
operators from $H^{p(\cdot)}_{A}$ to $H^{p(\cdot)}_{A}$,
we need to prove that anisotropic Calder¡äon-Zygmund operators T map atoms into harmless constant multiples of molecules. Generally, we cannot expect this unless we also assume that the considered operators preserve vanishing moments, which is given in the following definition introduced by Bownik (\cite[Definition 9.4]{b03}).
\begin{definition}
We say that an anisotropic Calder\'on-Zygmund operator of order $N$ satisfies $T^{*}(x^{\gamma})=0$ for all $|\gamma|\leq s$, where $s<N\ln \lambda_-/\ln \lambda_+$, if for any $f\in L^q$ with compact support, $q\in[1,\,\infty]\cap(p_+,\,\infty]$ with $p_+$ as in \eqref{e2.5},
 and
$$\int_{\rn}x^{\alpha}f(x)dx=0\,\,\mathrm{for}\,\,\mathrm{all}\,\,|\alpha|< N,$$
we also have
$$\int_{\rn}x^{\gamma}Tf(x)dx=0\,\,\mathrm{for}\,\,\mathrm{all}\,\,|\gamma|\leq s.$$
\end{definition}

Notice that Definition 4.5 coincides with the analogous property in the isotropic
setting investigated by Meyer in [36, Chapter 7.4]. Furthermore, the condition $T^\ast(x^\az)=0$ is automatically satisfied when $T$ is a Calder\"on-Zygmund operator with convolutional kernel (see [17, Chapter III.7]).

The main results of this section are the following two theorems.

\begin{theorem}\label{t4.1}
Let $p(\cdot)\in C^{\log}$ and $\beta:=\lf(\frac{\ln b}{\ln \lambda_-}+s+1\r)\frac{\ln \lambda_-}{\ln b}$, where $s:=\lfloor(1/{p_-}-1) {\ln b/\ln \lambda_-}\rfloor$. If $N\in\nn$ satisfies $N\geq(\beta-1)\ln b/\ln{\lz_-}$ and $T$ is an anisotropic Calder\'on-Zygmund
operator of order $N$, then $T$ can be extended to a bounded linear operator from $H^{p(\cdot)}_{A}$ to $L^{p(\cdot)}$. Moreover, there exists a positive constant $C$ such that, for all
$f\in H^{p(\cdot)}_{A}$,
\begin{eqnarray}\label{e4.3}
\|T(f)\|_{L^{p(\cdot)}}\leq C\, \|f\|_{H^{p(\cdot)}_{A}}.
\end{eqnarray}
\end{theorem}

\begin{theorem}\label{t4.2}
Let $p(\cdot)\in C^{\log}$. If $N\in\nn$ and $T$ is an anisotropic Calder\'on-Zygmund
operator of order $N$, then $T$ can be extended to a bounded linear operator from $H^{p(\cdot)}_{A}$ to $H^{p(\cdot)}_{A}$, provided that $T^*(x^\alpha)=0$ for all $\alpha\in\zz^n_+$ with $|\alpha|\le s$, where $s\in[\lfloor(1/{p_-}-1) {\ln b/\ln \lambda_-}\rfloor,\,\infty)\cap\zz_+$ with $p_-$ as in \eqref{e2.5} and $s<N\ln \lambda_-/\ln \lambda_+$. Moreover, there exists a positive constant $C$ such that, for all
$f\in H^{p(\cdot)}_{A}$,
\begin{eqnarray}\label{e4.4x}
\|T(f)\|_{H^{p(\cdot)}_{A}}\leq C\, \|f\|_{H^{p(\cdot)}_{A}}.
\end{eqnarray}
\end{theorem}
\begin{remark}
\begin{enumerate}
\item[(i)]
When the exponent function $p(\cdot)$ is reduced to the constant exponent $p$, i.e., $p(\cdot):=p\in(0,\,1]$, we have $p_-=p$ and the space
$H^{p(\cdot)}_{A}$ is reduced to the anisotropic Hardy space $H_A^p$ and now
Theorems \ref{t4.1} and \ref{t4.2} coincide with \cite[Theorem 9.9, Theorem 9.8]{b03} of Bownik, respectively.
\item[(ii)]When $A:=2{\rm I}_{n\times n}$ and $T$ is  a Calder¨®n-Zygmund operator
of convolution type, Theorems \ref{t4.1} and \ref{t4.2} coincide with \cite[Proposition 5.3, Theorem 5.5]{ns12} of Nakai and Sawano, respectively.
\item[(iii)]When $A:=2{\rm I}_{n\times n}$ and $\rho(x):=|x|^n$ for all $x\in\rn$,
as applications of
Theorems \ref{t4.1} and \ref{t4.2}, we obtain that the Hilbert transform and the Riesz transforms are bounded from $H^{p(\cdot)}_{A}$ to $L^{p(\cdot)}$ and from $H^{p(\cdot)}_{A}$ to itself, whose proof is similar to that of \cite[Theorem 4.7]{lql18}.
\end{enumerate}
\end{remark}

We now turn to the proofs of Theorems \ref{t4.1} and \ref{t4.2}. To this end, let us begin with some lemmas.
The following Lemma \ref{l4.1yx}, Lemma \ref{l4.1x}, Lemma \ref{l4.2x}, Lemma \ref{l4.2} and Lemma \ref{l4.1xx}, respectively, come from
 \rm{\cite[Lemma 5.2, Lemma 4.6, Lemma 4.7]{lyy17x}}, \rm{\cite[Lemma 2.3]{blyz10}} and \rm{\cite[Lemma 2.71]{cf13}}.

\begin{lemma}\label{l4.1yx}
 Let $p(\cdot)\in C^{\log}$, $q \in (1,\,\infty]$ and $s$ be as in \eqref{e2.10}. Then, for any $f\in H^{p(\cdot)}_{A}\cap L^q$, there exist $\{\lambda_i\}_{{i\in\nn}}\subset\mathbb{C}$, dilated balls $\{x_i+B_{\ell_i}\}_{i\in\nn}\subset\mathfrak{B}$ and
 $(p(\cdot),\,\infty,\,s)-$atoms $\{a_i\}_{{i\in\nn}}$ such that
 $$f=\sum_{i\in\nn} \lz_{i}a_i,$$
 where the series converge both almost everywhere and in $\cs'$,
 $$\mathrm{supp}\, a_i\subset x_i + B_{\ell_i}$$
and
$$\lf\|\lf\{\sum_{i\in\nn} \lf[\frac{|\lambda_i|\chi_{x_i+ B_{{\ell_i}}}}{\lf\|\chi_{x_i+ B_{{\ell_i}}}\r\|_{L^{p(\cdot)}}}\r]^{\underline{p}}\r\}^{1/\underline{p}}\r\|
_{L^{p(\cdot)}}\lesssim\|f\|_{H^{p(\cdot)}_{A}}$$
with $\underline{p}$ as in \eqref{e2.5.1}.
\end{lemma}

\begin{lemma}\label{l4.2}
Let $A$ be a dilation. Then there exists a collection
$$\mathcal{Q}:=\lf\{Q^k_\alpha\subset\rn:k\in\zz,\,\alpha\in I_k\r\}$$
of open subsets, where $I_k$ is certain index set, such that
\begin{enumerate}
\item[\rm{(i)}] for any $k\in\zz$, $|\rn \backslash \cup_\alpha Q^k_\alpha|=0$ and, when $\alpha\neq\beta$, $Q^k_\alpha\cap Q^k_\beta=\emptyset$;
\item[\rm{(ii)}] for any $\alpha,\beta,k,\ell$ with $\ell\geq k$, either $Q^k_\alpha\cap Q^\ell_\beta=\emptyset$ or
$Q^\ell_\alpha\subset Q^k_\beta$;
\item[\rm{(iii)}] for each $(\ell,\,\beta)$ and each $k<\ell$, there exists a unique $\alpha$ such that
$Q^\ell_\beta\subset Q^k_\alpha$;
\item[\rm{(iv)}] there exists some negative integer $v$ and positive integer $u$ such that, for any $Q^k_\alpha$ with
$k\in\zz$ and $\alpha\in I_k$, there exists $x_{Q^k_\alpha}\in Q^k_\alpha$ satisfying that, for any $x\in Q^k_\alpha$,
$$x_{Q^k_\alpha}+B_{vk-u}\subset Q^k_\alpha\subset x+B_{vk+u}.$$
\end{enumerate}
\end{lemma}
In what follows, we call $\mathcal{Q}:=\{Q^k_\alpha\}_{k\in\zz,\,\alpha\in I_k}$ from Lemma \ref{l4.2} dyadic cubes and $k$ the level,
denoted by $\ell(Q^k_\alpha)$, of the dyadic cube $Q^k_\alpha$ with $k\in\zz$ and $\alpha\in I_k$.
\begin{remark}\label{r4.1}
In the definition of $(p(\cdot),\,q,\,s)$-atoms (see Definition \ref{d3.1}), if we replace dilated balls $\mathfrak{B}$
by dyadic cubes, then, from Lemma \ref{l4.2}, we deduce that the corresponding variable anisotropic atomic Hardy
space coincides with the original one (see Definition \ref{d2.5}) in the sense of equivalent quasi-norms.
\end{remark}
\begin{lemma}\label{l4.1x}
Let $p(\cdot)\in C^{\log}$ and $q\in(1,\,\infty]\cap(p_{+},\,\infty]$ with $p_+$ as in \eqref{e2.5}. Assume that
$\{\lambda_{i}\}_{i\in\nn}\subset\ccc$, $\{B^{(i)}\}_{i\in\nn}\subset\mathfrak{B}$ and $\{a_{i}\}_{i\in\nn}\in L^q$
satisfy, for any $i\in\mathbb{N}$, \,$\mathrm{supp}\, a_i\subset B^{(i)}$,
$$\|a\|_{L^q}\le \frac{|B^{(i)}|^{1/q}}{\|\chi_{B^{(i)}}\|_{L^{p(\cdot)}}}$$
and
$$\lf\|\lf\{\sum_{i\in\nn} \lf[\frac{|\lambda_i|\chi_{{B^{(i)}}}}{\|\chi_{{B^{(i)}}}\|_{L^{p(\cdot)}}}\r]^{\underline{p}}\r\}^{1/\underline{p}}\r\|
_{L^{p(\cdot)}}<\infty.$$
Then
$$\lf\|\lf[\sum_{i\in\nn}|\lambda_{i}a_{i}|^{\underline{p}}\r]^{1/\underline{p}}\r\|
_{L^{p(\cdot)}}\leq C\lf\|\lf\{\sum_{i\in\nn} \lf[\frac{|\lambda_i|\chi_{{B^{(i)}}}}{\|\chi_{{B^{(i)}}}\|_{L^{p(\cdot)}}}\r]^{\underline{p}}\r\}^{1/\underline{p}}\r\|
_{L^{p(\cdot)}},$$
where $\underline{p}$ is as in \eqref{e2.5.1} and $C$ is a positive constant independent of $\{\lambda_{i}\}_{i\in\nn}$, $\{B^{(i)}\}_{i\in\nn}$ and $\{a_{i}\}_{i\in\nn}$.
\end{lemma}
\begin{lemma}\label{l4.2x}
 Let $p(\cdot)\in C^{\log}$ and $q\in(1,\,\infty]\cap(p_{+},\,\infty]$ with $p_+$ as in \eqref{e2.5}. Then $H^{p(\cdot)}_{A}\cap L^q$ is dense in $H^{p(\cdot)}_{A}$.
\end{lemma}
\begin{lemma}\label{l4.1xx}
Given $p(\cdot)\in\mathcal{P}(\rn)$, $L^{p(\cdot)}(\rn)$ is complete: every Cauchy sequence in $L^{p(\cdot)}(\rn)$ converges in norm.
\end{lemma}

\begin{lemma}\label{l4.2xx}
 Let $p(\cdot)\in C^{\log}$. Then $H^{p(\cdot)}_{A}$ is complete: every Cauchy sequence in $H^{p(\cdot)}_{A}$ converges in norm.
\end{lemma}
\begin{proof}
We show this lemma by borrowing some ideas from the proof of \rm{\cite[Proposition 4.1]{cw14}}. To prove that $H^{p(\cdot)}_{A}$ is complete, it is sufficient to prove that if $\{f_k\}_{k=1}^{\infty}$ is a sequence in $H^{p(\cdot)}_{A}$ such that
$$\sum_{k\in\mathbb{N}}\|f_k\|_{H^{p(\cdot)}_{A}}^{\underline{p}}<\infty,$$
where $\underline{p}$ is as in \eqref{e2.5.1}, then the series $\sum_{k\in\mathbb{N}}f_k$ in $H^{p(\cdot)}_{A}$ converges in norm.

For any $j\in\mathbb{N}$, let $F_j:=\sum_{k=1}^jf_k$. From Remark \ref{r2.1}(i) and the fact that $\underline{p}$ is as in \eqref{e2.5.1}, we conclude that, for any $m, n\in\mathbb{N}$ with $m>n$,
\begin{align*}
\|F_m-F_n\|_{H^{p(\cdot)}_{A}}^{\underline{p}}=&\lf\|\sum_{k=n+1}^m f_k\r\|_{H^{p(\cdot)}_{A}}^{\underline{p}}\leq\lf\|\sum_{k=n+1}^m M_N(f_k)\r\|_{L^{p(\cdot)}}^{\underline{p}}\\
=&\lf\|\lf[\sum_{k=n+1}^m M_N(f_k)\r]^{\underline{p}}\r\|_{L^{{p(\cdot)}/\underline{p}}}
\leq\lf\|\sum_{k=n+1}^m\lf[M_N(f_k)\r]^{\underline{p}}\r\|_{L^{{p(\cdot)}/\underline{p}}}\\
\leq&\sum_{k=n+1}^m\lf\|\lf[M_N(f_k)\r]^{\underline{p}}\r\|_{L^{{p(\cdot)}/\underline{p}}}
=\sum_{k=n+1}^m\lf\|M_N(f_k)\r\|_{L^{{p(\cdot)}}}^{\underline{p}}\\
=&\sum_{k=n+1}^m\lf\|f_k\r\|_{H^{{p(\cdot)}}_{A}}^{\underline{p}}.
\end{align*}
By this, we know that $\{F_j\}_{j\in \mathbb{N}}$ is a Cauchy sequence in $H^{p(\cdot)}_{A}$. Since $H^{p(\cdot)}_{A}$ is continuously contained in
$\mathcal{S}^{'}$(see \rm{\cite[Lemma 4.3]{lyy17x}}), thus $\{F_j\}_{j\in \mathbb{N}}$ is a Cauchy sequence in $\mathcal{S}^{'}$. Therefore we know that there exists a
tempered distribution $f\in\mathcal{S}^{'}$ such that
$F_j\rightarrow f$ in $\mathcal{S}^{'}$ \,as\, $j\rightarrow\infty$.

Next we prove $f\in H^{p(\cdot)}_{A}$. Since
$$M_N(f)\leq\lim_{j\rightarrow\infty}\sum_{k=1}^jM_N(f_k),$$
by Remark \ref{r2.1}(i), we obtain
\begin{align*}
\lf\|M_N(f)\r\|_{L^{{p(\cdot)}}}^{\underline{p}}&\leq\lf\|\lim_{j\rightarrow\infty}\sum_{k=1}^j M_N(f_k)\r\|_{L^{{p(\cdot)}}}^{\underline{p}}
=\lim_{j\rightarrow\infty}\lf\|\sum_{k=1}^jM_N(f_k)\r\|_{L^{{p(\cdot)}}}^{\underline{p}}\\
&\leq\lim_{j\rightarrow\infty}\lf\|\sum_{k=1}^j[M_N(f_k)]^{\underline{p}}\r\|_{L^{{p(\cdot)}/\underline{p}}}
=\sum_{k=1}^\infty\lf\|[M_N(f_k)]^{\underline{p}}\r\|_{L^{{p(\cdot)}/\underline{p}}}\\
&=\sum_{k=1}^\infty\lf\|f_k\r\|_{H^{{p(\cdot)}}_{A}}^{\underline{p}}<\infty,
\end{align*}
which implies that $f\in H^{p(\cdot)}_{A}$.

Finally, by Remark \ref{r2.1}(i), we find that
\begin{align*}
\lf\|f-\sum_{k=1}^jf_k\r\|_{H^{{p(\cdot)}}_{A}}^{\underline{p}}&=\lf\|\lim_{s\rightarrow\infty}\sum_{k=j+1}^{s}f_k\r\|_{H^{{p(\cdot)}}_{A}}^{\underline{p}}
\leq\lim_{s\rightarrow\infty}\lf\|\sum_{k=j+1}^{s}f_k\r\|_{H^{{p(\cdot)}}_{A}}^{\underline{p}}\\
&\leq\sum_{k=j+1}^{\infty}\lf\|f_k\r\|_{H^{{p(\cdot)}}_{A}}^{\underline{p}}\rightarrow 0, \,\,\mathrm{as}\,\,j\rightarrow\infty,
\end{align*}
which implies that $\{F_j\}_{j\in\nn}$ in $H^{p(\cdot)}_{A}$ converges to $f=\sum_{k\in\mathbb{N}}f_k$ in norm. This finishes the proof of Lemma \ref{l4.2xx}.
\end{proof}

Now we prove Theorem \ref{t4.1}.
\begin{proof}[Proof of Theorem \ref{t4.1}]
First, we show that \eqref{e4.3} holds true for any $f\in H^{p(\cdot)}_{A}\cap L^r$ with $r\in(1,\,\infty]\cap(p_{+},\,\infty]$. For $f\in H^{p(\cdot)}_{A}\cap L^r$,
from Lemma \ref{l4.1yx} and Reamrk \ref{r4.1},
we know that there exist $\{\lambda_i\}_{i\in\nn}\subset\ccc$ and a sequence of
$(p(\cdot),\,\infty,\,s)$-atoms, ${\{a_i\}}_{i\in\nn}$, supported, respectively,
on ${\{Q_{i}}\}_{i\in\nn}\subset\mathcal{Q}$ such that
\begin{align*}
f=\sum_{i\in\nn} \lz_{i}a_i \ \ \mathrm{in\ } \ \cs' \,\,\mathrm{and}\,\,\mathrm{almost}\,\, \mathrm{everywhere},
\end{align*}
and
\begin{eqnarray}\label{e4.51}
 \lf\|\lf\{\sum_{i\in\nn} \lf[\frac{|\lambda_i|\chi_{Q_{{i}}}}{\lf\|\chi_{Q_{{i}}}\r\|_{L^{p(\cdot)}}}\r]^{\underline{p}}\r\}^{1/\underline{p}}\r\|
_{L^{p(\cdot)}}\lesssim\|f\|_{H^{p(\cdot)}_{A}}.
\end{eqnarray}

Let $\omega:=u+v+2\sigma$ with $u$ and $v$ as in Lemma \ref{l4.2}. Then, by Remark \ref{r2.1}(i), we find that, for any $x\in\rn$,
\begin{align}\label{e4.52}
\|T(f)\|_{L^{p(\cdot)}}^{\underline{p}}&=\lf\|T\lf(\sum_{i\in \nn}\lambda_{i}a_{i}\r)\r\|_{L^{p(\cdot)}}^{\underline{p}}
=\lf\|\sum_{i\in \nn}\lambda_{i}T(a_{i})\r\|_{L^{p(\cdot)}}^{\underline{p}}\\\nonumber
&\leq\lf\|\sum_{i\in \nn}|\lambda_{i}|T(a_{i})\chi_{{A^{\omega} Q_{{i}}}}\r\|_{L^{p(\cdot)}}^{\underline{p}}
+\lf\|\sum_{i\in \nn}|\lambda_{i}|T(a_{i})\chi_{({A^{\omega} Q_{{i}}})^{\complement}}\r\|_{L^{p(\cdot)}}^{\underline{p}}\\\nonumber
&\lesssim\lf\|\lf\{\sum_{i\in \nn}\lf[|\lambda_{i}|T(a_{i})\chi_{A^{\omega} Q_{{i}}}\r]^{\underline{p}}\r\}^{1/\underline{p}}\r\|_{L^{p(\cdot)}}^{\underline{p}}
+\lf\|\sum_{i\in \nn}|\lambda_{i}|T(a_{i})\chi_{({A^{\omega} Q_{{i}}})^{\complement}}\r\|_{L^{p(\cdot)}}^{\underline{p}}\\\nonumber
&=:\mathrm{J_{1}}+\mathrm{J_{2}},\nonumber
\end{align}
where  $A^{\omega} Q_i$ is the $A^{\omega}$ concentric expanse on $Q_i$, such that
$x_{Q_i}+A^{\omega}B_{v\ell(Q_i)-u}\subset A^{\omega}Q_i \subset x_{Q_i}+A^{\omega} B_{v\ell(Q_i)+u}$.

For the term $\mathrm{J_1}$, by the boundedness of $T$ on $L^q$ with $q\in(\max\{p_+,\,1\},\,\infty)$(see Theorem \ref{t4.0}), Lemma \ref{l4.1x} and \eqref{e4.51}, we conclude that
\begin{align*}
\mathrm{J_1}&\lesssim\lf\|\lf\{\sum_{i\in\nn} \lf[\frac{|\lambda_i|\chi_{{Q_{{i}}}}}{\lf\|\chi_{{ Q_{{i}}}}\r\|_{L^{p(\cdot)}}}\r]^{\underline{p}}\r\}^{1/\underline{p}}\r\|
_{L^{p(\cdot)}}^{\underline{p}}
\lesssim\|f\|_{H^{p(\cdot)}_{A}}^{\underline{p}}.
\end{align*}

To deal with ${\mathrm{J_{2}}}$, assume that $a_i(x)$ is a $(p(\cdot),\,\infty,\,s)$-atom supported on a dyadic cube $Q_{{i}}$.
From the size condition of $a_i(x)$ and similar to that of $(4.13)$ in the proof of \cite[Theorem 4.4]{lql18}, we conclude that, for any $x\in(A^\omega Q_{{i}})^\complement$,
\begin{eqnarray}\label{e4.13x}|T(a_i)(x)|
\ls \lf\|\chi_{Q_{{i}}}\r\|^{-1}_{L^{p(\cdot)}}\lf[M_{HL}(\chi_{Q_{{i}}})(x)\r]^\beta.
\end{eqnarray}
By \eqref{e4.13x} and an argument same as that used in the proof of \eqref{e3.10x}, we obtain
\begin{align*}
\mathrm{J_2}&\lesssim\lf\|\lf\{\sum_{i\in\nn} \lf[\frac{|\lambda_i|\chi_{{Q_{{i}}}}}{\lf\|\chi_{{ Q_{{i}}}}\r\|_{L^{p(\cdot)}}}\r]^{\underline{p}}\r\}^{1/\underline{p}}\r\|
_{L^{p(\cdot)}}^{\underline{p}}
\lesssim\|f\|_{H^{p(\cdot)}_{A}}^{\underline{p}}.
\end{align*}
Combining \eqref{e4.52} and the estimates of $\mathrm{J_1}$ and $\mathrm{J_2}$, we further conclude that
$$\|T(f)\|_{L^{p(\cdot)}}
\lesssim\|f\|_{H^{p(\cdot)}_{A}}.$$

Next, we prove that \eqref{e4.3} also holds true for any $f\in H^{p(\cdot)}_{A}$.
Let $f\in H^{p(\cdot)}_{A}$, by Lemma \ref{l4.2x}, we know that there exists a sequence
$\{f_j\}_{j\in {\mathbb{Z}_+}}\subset H^{p(\cdot)}_A\cap L^r$ with $r\in(1,\,\infty]\cap(p_{+},\,\infty]$ such that $f_j \rightarrow f$ as $j\rightarrow\infty$ in $H^{p(\cdot)}_A$.
Therefore, $\{f_j\}_{j\in {\mathbb{Z}_+}}$ is a Cauchy sequence in $H^{p(\cdot)}_A$. By this, we see that, for any $j$, $k\in {\mathbb{Z}_+}$,
$$\|T(f_j)-T(f_k)\|_{L^{p(\cdot)}}=\|T(f_j-f_k)\|_{L^{p(\cdot)}}\lesssim\|f_j-f_k\|_{H^{p(\cdot)}_A}.$$
Notice that $\{T(f_j)\}_{j\in\mathbb{Z}_+}$ is also a Cauchy sequence in $L^{p(\cdot)}$. Applying Lemma \ref{l4.1xx}, we conclude that there exist a $g \in L^{p(\cdot)}$ such that $T(f_j)\rightarrow g$ as $j\rightarrow\infty$ in $L^{p(\cdot)}$.
Let $T(f):=g$. We claim that $T(f)$ is well defined. Indeed, for any other sequence $\{h_j\}_{j\in\mathbb{Z}_+}\subset H^{p(\cdot)}_A\cap L^r$ satisfying $h_j\rightarrow f$ as $j\rightarrow\infty$ in $H^{p(\cdot)}_A$, by Remark \ref{r2.1}(i), we have
\begin{align*}
\|T(h_j)-T(f)\|_{L^{p(\cdot)}}^{\underline{p}}&\leq\|T(h_j)-T(f_j)\|_{L^{p(\cdot)}}^{\underline{p}}+\|T(f_j)-g\|_{L^{p(\cdot)}}^{\underline{p}}.\\
&\lesssim\|h_j-f_j\|_{H^{p(\cdot)}_A}^{\underline{p}}+\|T(f_j)-g\|_{L^{p(\cdot)}}^{\underline{p}}\\
&\lesssim\|h_j-f\|_{H^{p(\cdot)}_A}^{\underline{p}}+\|f-f_j\|_{H^{p(\cdot)}_A}^{\underline{p}}+\|T(f_j)-g\|_{L^{p(\cdot)}}^{\underline{p}}\rightarrow 0 \,\,\mathrm{as}\,\, j\rightarrow 0,
\end{align*}
which is wished.

From this, we see that, for any $f\in H^{p(\cdot)}_A$,
$$\|T(f)\|_{L^{p(\cdot)}}=\|g\|_{L^{p(\cdot)}}=\lim_{j\rightarrow\infty}\|T(f_j)\|_{L^{p(\cdot)}}\lesssim\lim_{j\rightarrow\infty}\|f_j\|_{H^{p(\cdot)}_A} \thicksim\|f\|_{H^{p(\cdot)}_A},$$
which implies that \eqref{e4.3} also holds true for any $f\in H^{p(\cdot)}_{A}$ and hence completes the proof of Theorem \ref{t4.1}.
\end{proof}

To prove Theorem \ref{t4.2}, we need the following technical lemma which is just \cite[Lemma 4.10]{lql18}.

\begin{lemma}\label{l4.3}
Let $p(\cdot)\in C^{\log}$, $q\in(1,\,\infty]$ and
$s\in[\lfloor(1/{p_-}-1) {\ln b/\ln \lambda_-}\rfloor,\,\infty)\cap\zz_+$ with $p_-$ as in \eqref{e2.5}.
Suppose that $T$ is an anisotropic Calder\'{o}n-Zygmund operator of order $N\in\nn$ satisfying
$$T^*(x^\alpha)=0$$
for all $|\az|\le s$ with $s<N\ln \lambda_-/\ln \lambda_+$.
Then, for any $(p(\cdot),\,q,\,s)$-atom $a$ supported on some $x_0+B_{j_0}$
with $x_0\in\rn$ and $j_0\in\zz$, $T(a)$ is a harmless constant multiple
of a $(p(\cdot),\,q,\,s,\,\varepsilon)$-molecule associated with $x_0+B_{j_0}$, where
$\ez:=N\ln{\lz_-}/\ln b+1/q'$.
\end{lemma}

\begin{proof}[Proof of Theorem \ref{t4.2}]
First, we show that \eqref{e4.4x} holds true for any $f\in H^{p(\cdot)}_{A}\cap L^r$ with $r\in(1,\,\infty]\cap(p_{+},\,\infty]$. For any $f\in H^{p(\cdot)}_{A}\cap L^r$, by Lemma \ref{l4.1yx},
we know that there exist $\{\lambda_i\}_{i\in\nn}\subset\ccc$ and a sequence of
$(p(\cdot),\,q,\,s)$-atoms, $\{a_i\}_{i\in\nn}$, supported, respectively,
on $\{{B^{(i)}}\}_{i\in\nn}\subset\mathfrak{B}$, where $B^{(i)}:=x_i+B_{\ell_i}$ with $x_i\in\rn$ and $\ell_i\in\zz$, such that
\begin{align*}
f=\sum_{i\in\nn} \lz_{i}a_i \ \ \mathrm{in\ } \ \cs' \ \ \mathrm{and}\ \ \mathrm{almost}\ \ \mathrm{everywhere},
\end{align*}
and
\begin{eqnarray}\label{e4.21}
\|f\|_{H^{p(\cdot),\,q,\,s}_{A}}
\thicksim \lf\|\lf\{\sum_{i\in\nn} \lf[\frac{|\lambda_i|\chi_{{B^{(i)}}}}{\|\chi_{{B^{(i)}}}\|_{L^{p(\cdot)}}}\r]^{\underline{p}}\r\}^{1/\underline{p}}\r\|
_{L^{p(\cdot)}}.
\end{eqnarray}
It is easy to see that, for all $N\in\mathbb{N}\cap[\lfloor({1/\underline{p}}-1)/\ln\lambda_{-}\rfloor+2,\,\infty)$ and $x\in\rn$,
\begin{align}\label{e4.52x}
&\lf\|M_{N}(T(f))\r\|_{L^{p(\cdot)}}^{\underline{p}}\\\nonumber=&\lf\|M_{N}\lf(T\lf(\sum_{i\in \nn}\lambda_{i}a_{i}\r)\r)\r\|_{L^{p(\cdot)}}^{\underline{p}}\\\nonumber
\leq&\lf\|\sum_{i\in \nn}|\lambda_{i}|M_{N}(T(a_{i}))\r\|_{L^{p(\cdot)}}^{\underline{p}}\\\nonumber
\leq&\lf\|\sum_{i\in \nn}|\lambda_{i}|M_{N}(T(a_{i}))\chi_{A^{\omega}B^{(i)}}\r\|_{L^{p(\cdot)}}^{\underline{p}}+\lf\|\sum_{i\in \nn}|\lambda_{i}|M_{N}(T(a_{i}))\chi_{({A^{\omega}B^{(i)}})^{\complement}}\r\|_{L^{p(\cdot)}}^{\underline{p}}\\\nonumber
\lesssim&\lf\|\lf\{\sum_{i\in \nn}\lf[|\lambda_{i}|M_{N}(T(a_{i}))\chi_{A^{\omega}B^{(i)}}\r]^{\underline{p}}\r\}^{1/\underline{p}}\r\|_{L^{p(\cdot)}}^{\underline{p}}+\lf\|\sum_{i\in \nn}|\lambda_{i}|M_{N}(T(a_{i}))\chi_{({A^{\omega}B^{(i)}})^{\complement}}\r\|_{L^{p(\cdot)}}^{\underline{p}}\\\nonumber
=&:\mathrm{K_{1}}+\mathrm{K_{2}},\nonumber
\end{align}
where $A^{\omega}B^{(i)}$ is the $A^{\omega}$ concentric expanse on $B^{(i)}$, that is, $A^{\omega}B^{(i)}:=x_i+A^{\omega}B_{\ell_i}$,
 $\omega:=u+v+2\sigma$ with $u$ and $v$ as in Lemma \ref{l4.2} and $\underline{p}$ as in \eqref{e2.5.1}.

For ${\mathrm{K_1}}$, from the fact that $M_{N}$ and $T$ are bounded on $L^q$ for all $q\in(1,\,\infty)$ (see Theorem \ref{t4.0}),
we know that
$$\lf\|M_{N}(T(a_i))\chi_{A^{\omega}B^{(i)}}\r\|_{L^q}\lesssim\|a_i\chi_{A^{\omega}B^{(i)}}\|_{L^q}\lesssim\frac{|B^{(i)}|^{1/q}}{\|\chi_{B^{(i)}}\|_{L^{p(\cdot)}}}.$$
From this, Lemma \ref{l4.1x} and \eqref{e4.21}, we further deduce that
\begin{align*}
\mathrm{K_1}&\lesssim\lf\|\lf\{\sum_{i\in\nn} \lf[\frac{|\lambda_i|\chi_{{B^{(i)}}}}{\|\chi_{{B^{(i)}}}\|_{L^{p(\cdot)}}}\r]^{\underline{p}}\r\}^{1/\underline{p}}\r\|
_{L^{p(\cdot)}}^{\underline{p}}
\thicksim\|f\|_{H^{p(\cdot),\,q,\,s}_{A}}^{\underline{p}}.
\end{align*}

For ${\mathrm{K_2}}$, for any $i\in\nn$, and $x\in ({A^\omega B^{(i)}})^\complement$, by Lemma \ref{l4.3}, for any $(p(\cdot),\,q,\,s)$-atom $a_i(x)$ supported on a ball $B^{(i)}$, we see that $T(a_i)$ is a harmless constant multiple
of a $(p(\cdot),\,q,\,s,\,\varepsilon)$-molecule associated with $B^{(i)}$, where
$\ez:=N\log_b(\lz_-)+1/q'$. From this and an argument similar to that used in the proof of \eqref{e3.12}, we know that
\begin{eqnarray}\label{e4.13}
M_N(T(a_i))(x)
\ls\lf\|\chi_{B^{(i)}}\r\|^{-1}_{L^{p(\cdot)}}
\lf[M_{HL}(\chi_{B^{(i)}})(x)\r]^\beta,
\end{eqnarray}
where, for any $i\in\nn$, $x_i$ denotes the centre of the dilated ball $B^{(i)}$ and
$\beta$ as in \eqref{e3.12.1}.
By \eqref{e4.13} and an argument same as that used in the proof of \eqref{e3.10x}, we obtain
\begin{align*}
\mathrm{K_2}&\lesssim\lf\|\lf\{\sum_{i\in\nn} \lf[\frac{|\lambda_i|\chi_{{B^{(i)}}}}{\|\chi_{{B^{(i)}}}\|_{L^{p(\cdot)}}}\r]^{\underline{p}}\r\}^{1/\underline{p}}\r\|
_{L^{p(\cdot)}}^{\underline{p}}
\thicksim\|f\|_{H^{p(\cdot),\,q,\,s}_{A}}^{\underline{p}}.
\end{align*}
Combining \eqref{e4.52x} and the estimates of $\mathrm{K_1}$ and $\mathrm{K_2}$, we further conclude that
$$\|T(f)\|_{H^{p(\cdot)}_A}
\lesssim\|f\|_{H^{p(\cdot),\,q,\,s}_{A}}\thicksim \|f\|_{H^{p(\cdot)}_{A}}.$$

From Lemma \ref{l4.2xx}, Remark \ref{r2.1}(i) and a similar proof of Theorem \ref{t4.1}, we further conclude that
\eqref{e4.4x} also holds true for any $f\in H^{p(\cdot)}_{A}$.
This finishes the proof of Theorem \ref{t4.2}.
\end{proof}

\textbf{Acknowledgements.} The authors would like to express their
deep thanks to the referees for their very careful reading and useful
comments which do improve the presentation of this article.

\bigskip

\noindent Wenhua Wang, Xiong Liu, Aiting Wang and Baode Li
\medskip

\noindent
College of Mathematics and System Sciences \\
Xinjiang University\\
Urumqi 830046\\
P. R. China

\smallskip

\noindent{E-mail }:\\
\texttt{1663434886@qq.com} (Wenhua Wang)\\
\texttt{1394758246@qq.com} (Xiong Liu)  \\
\texttt{2358063796@qq.com} (Aiting Wang) \\
\texttt{1246530557@qq.com} (Baode Li)

\bigskip \medskip

\end{document}